\documentclass[fleqn,preprint,12pt,3p]{elsarticle}

\usepackage{amsmath}
\usepackage{amssymb}
\usepackage{amsthm}

\usepackage{algorithm}
\usepackage{algorithmic}
\usepackage{pseudocode}
\usepackage{multirow}
\usepackage{booktabs}
\usepackage{subfig}
\usepackage{amsmath,amssymb,amsthm}
\usepackage{epstopdf}
\usepackage{ulem}
\usepackage[colorlinks]{hyperref}

\usepackage{graphicx}
\usepackage{epsfig}
\newtheorem{theorem}{Theorem}[section]
\newtheorem{lemma}[theorem]{Lemma}

\newtheorem{proposition}[theorem]{Proposition}

\newdefinition{definition}[theorem]{Definition}
\newdefinition{remark}[theorem]{Remark}
\newdefinition{example}[theorem]{Example}
\journal{}

\begin{document}

\begin{frontmatter}



\title{Sparse recovery guarantees for block orthogonal binary matrices constructed via Generalized Euler Squares}


\author{Pradip Sasmal$^\dagger$, Phanindra Jampana and C. S. Sastry}
\address{$^\dagger$Dept. of ECE, Indian Institute of Science, Bangalore, India\\ email:~$^\dagger$\{pradipsasmal\}@iisc.ac.in, \\
Indian Institute of Technology, Hyderabad, Telangana. 502285, India. \\ Email:\{pjampana, csastry\}@iith.ac.in}

%\author{Pradip Sasmal\footnote{author for correspondence Telephone: 091-40-23016072, Fax: 091-40-23016032 } R. Ramu Naidu, and C. S. Sastry \\ Department of Mathematics \\ Indian Institute of Technology, Hyderabad-502205, India \\
%Email: \{ma12p005, ma11p003, csastry\} @ iith.ac.in \\ P. V. Jampana \\ Department of Chemical Engineering\\ Indian Institute of Technology, Hyderabad-502205, India \\
%Email: pjampana@ iith.ac.in}

%\author{Pradip Sasmal\footnote{Author for correspondence Telephone: 091-40-23016072, Fax: 091-40-23016032 }, R. Ramu Naidu and C. S. Sastry}
%\address{Department of Mathematics \\ Indian Institute of Technology, Hyderabad-502205, India \\
%Email: \{ma12p005, ma11p003, csastry\} @ iith.ac.in}

%\author{P. V. Jampana}
%\address{Department of Chemical Engineering\\ Indian Institute of Technology, Hyderabad-502205, India \\
%Email: pjampana@ iith.ac.in}

\begin{abstract}
%Coherence and Restricted Isometry Property(RIP) play significant role in the field of Compressed Sensing(CS) to provide sparse representation of a signal. But deterministic construction of matrices with good coherence and RIP is found to be difficult. 
In recent times, the construction of deterministic matrices %with low coherence 
%, suitable for compressed sensing (CS),
has gained popularity as an alternative of random matrices as they provide guarantees for recovery of sparse signals. In particular, the construction of binary matrices has attained significance due to their potential for hardware-friendly implementation and appealing applications.
%via basic pursuit and orthogonal matching pursuit. 
%For the recovery of block sparse signals, on the other hand, low block coherence is required.  
 %as they provide guarantees . 
%But less attention has been paid to constructing deterministic matrices with low block coherence. It is known that matrices consisting of several orthogonal blocks are suitable for recovery of block sparse signal. 
%For a fixed row size, it is difficult to construct a matrix (CS matrix, in short) both with small coherence and huge column size. 
%Despite several recent results on  deterministic CS matrices, 
Our present work aims at constructing incoherent binary matrices consisting of orthogonal blocks with small block coherence. We show that the binary matrices constructed from Euler squares %\cite{ram_2016} 
 exhibit block orthogonality and possess low block coherence. %Furthermore, we define generalized Euler square (GES) and the binary matrices constructed from GES provide flexibility in row sizes and number of orthogonal blocks.  
% \sout{Recently binary CS matrices are constructed via Euler Square with more general row size}. 
 %The Euler Square based binary matrices, nevertheless, have a limitation in the sense that they possess not so good aspect ratio which limits their applicability in such applications as dimensionality reduction. 
 %This is due to an inherent property possessed by the Euler Squares. 
 With a goal of obtaining better aspect ratios, the present work generalizes the notion of Euler Squares and obtains a new class of deterministic binary matrices of more general size. For realizing the stated objectives, to begin with, the paper revisits the connection of finite field theory to Euler Squares and their construction. Using the stated connection, the work proposes Generalized Euler Squares (GES) and then presents a construction procedure. %Subsequently, we construct \sout{incoherent} 
 Binary matrices
 %\sout{for more general row sizes possessing}
 with low coherence and general row-sizes are obtained, whose column size is in the maximum possible order. Finally, the paper shows that the special structure possessed by GES is helpful in resulting in block orthogonal structure with small block coherence, which supports the recovery of block sparse signals.

\end{abstract}


\end{frontmatter}

% \linenumbers

%\section{}
%\label{}
\section{Introduction}
Recent developments at the intersection of algebra, probability and optimization theory, by the name of Compressed Sensing (CS), aim at providing sparse descriptions to linear systems.  These developments are found to have tremendous potential for several applications \cite{bourgain_2011}\cite{gil_2010}. Sparse representations of a vector are a powerful analytic tool in many application areas such as image/signal processing and numerical computation \cite{Bruckstein_2009}, to name a few. The need for sparse representation arises from the fact that several real life applications demand expressing data in terms of as few basis elements as possible. 
%The columns of the sensing matrix  are called atoms and the matrix $\Phi$ is often referred to as a dictionary. 
The developments of CS theory depend typically on sparsity and incoherence \cite{Bruckstein_2009}\cite{Kashin_2007}. 

%Sparsity expresses the idea that the “information rate” of a continuous time data may be much smaller than suggested by its bandwidth, or that a discrete-time data depends on a number of degrees of freedom which is comparably much smaller than its (finite) length. On the other hand, incoherence extends the duality between the time and frequency contents of data. 
For minimizing the computational complexity associated with the matrix-vector multiplication, it is desirable that a CS matrix has smaller density. Here, a CS matrix refers to a matrix that satisfies sparse recovery properties and density refers to the ratio of number of nonzero entries to the total number of entries of the matrix. Sparse CS matrices, especially binary matrices, contribute to fast processing with low computational complexity in CS \cite{gil_2010}.  
The low coherence of a CS matrix, on the other hand,  provides guarantees for recovering sparse signals via basis pursuit (BP) and orthogonal matching pursuit (OMP). In case of recovery of block sparse signals, the block coherence of a CS matrix plays an important role \cite{eldar_2010}. 
%It is desired to have small block coherence \cite{eldar_2010}. 
%Moreover the block orthogonality of a matrix comes handy in recovering block sparse signals.  So far, to our knowledge, a 
To date, less attention has been paid to constructing binary matrices with small coherence which also possess low block coherence. Our present work is on the construction of block orthogonal binary matrices with small coherence and block coherence so that they support recovery of sparse signals and block sparse signals simultaneously.  

%\sout{In recent literature \cite{ind_2008, adcock_2013, bryant_2014, mixon_2012, dima_2012}, deterministic construction of CS matrices has gained momentum. R. Devore \cite{Ronald_2007} has constructed deterministic binary sensing  matrix of size $p^2 \times p^{r+1}$, where $p$ is prime or prime power. The density of this matrix is $\frac{1}{p}$ and coherence is at most $\frac{r}{k}.$ S. Li et. al. \cite{li_2012} have generalized Devore's work, constructing binary sensing matrix of size $|\mathcal{P}|q \times q^{\mathcal{L}(G)}$, where $q$ is any prime power and  $\mathcal{P}$ is the set of all rational points on algebraic curve $\mathcal{X}$ over finite field $\mathbb{F}_q$. The density of this matrix is $\frac{1}{q}$. P. Indyk \cite{ind_2008} has constructed binary sensing matrices using hash functions and extractor graphs with sizes $r2^{O(\log\log n)^{O(1)}} \times n$, where $r \ll n$. In \cite{li_2014}, a binary matrix $\Phi$ of size $v\times \frac{v(v-1)}{k(k-1)}$ has been constructed by taking incidence matrix of a $(2,k,v)-$Steiner system.
 %A. Amini et. al. \cite{amini_2011} have constructed binary sensing matrices %with sizes $(16^{a}-1) \times \frac{(16^{a}-1)(16^{a}-6)}{5}$ 
%using OOC codes. The density of this matrix is $\frac{\lambda}{m}$, where $m$ is row size and $\lambda$ is the number of ones in each column. In \cite{lu_2012}, authors have constructed binary matrices from LDPC codes.} 
%\textcolor{red}{
In recent literature on deterministic CS matrices, several authors \cite{ind_2008},  \cite{li_2014}, \cite{lu_2012}, \cite{li_2012}, \cite{amini_2011}, \cite{adcock_2013}, \cite{bryant_2014}, \cite{mixon_2012}, \cite{dima_2012}  have made pioneering contributions using novel ideas. Nevertheless, in most of these constructions, the associated CS matrices have been constructed for sizes that are dictated by a certain family of prime numbers. Additionally, none of the these constructions are known to
%\textcolor{red}{(are known to ..?)} 
exhibit block orthogonal structure. %\textcolor{red}{(Pradip: Cross-check if it is true .. else this sentence can ruin this paper)}

%%%%%%%%%%%%%%%%%%%%%%%%%%%%%%%%%%%%%%%%%%%%%%%%%%%%%%%%%%%%%%%%%%%%%%%%%%%%%%%%%%%%%%%%%%%%%%%%%%%%%%%%%%%%%%%%%%%%%%%%%%%%%%%%%%%%%%%%%%%%%%%%%%%%%%%%%%%%%%%%%%%%%%%%%%%%%%%%%%%%%%%%%%%%%%%%%%%%%%%%%%%%%%%%%%%%%%%%%%%%%%%%%%%%%%%%%%%%%%%%%%%%%%%%%%%%%%%%%%%%%%%%%%%%%%%%%%%%%%%%%%%

Recently in \cite{ram_2016}, the authors have proposed a method to construct compressed sensing matrices from Euler Square. In particular, given a positive integer $m$ different from $p, p^2$ for a prime $p$, the  authors \cite{ram_2016} have shown that it is possible to construct a binary sensing matrix  of size $m \times c (m\mu)^2$ using Euler Squares, where $\mu$ is the coherence of the matrix (the maximum off-diagonal entry, in magnitude, of Gram matrix)  and $c \in [1,2).$   
%This construction results in general row-size CS matrices whose %\sout{Though we are able to construct matrix of more general row size via Euler Square, but} 
 %column sizes 
%\sout{of the matrix} 
%are smaller than the square of row sizes. {\color{blue}{In the construction of binary matrices in \cite{ram_2016}, the structure of Euler Squares has not been utilized. An Euler square of index $n,k$, namely, ES($n,k$), can be observed as an $n\times n$ square array of $k$-tuples where there is no intersection between any two $k-$tuples on the same row and same column, and there is at most one intersection between any two distinct $k-$tuples which are not from the same row or column. Now, whenever Euler Square of index $n,k$ exists,  we construct binary matrix with $n$ number of orthogonal blocks where each block is of size $nk\times n$ because there is no intersection between any two $k-$tuples on the same row.  
One of the objectives of present work is to improve upon the aspect ratio (the ratio of column size to row size) of the CS matrices. To realize this objective,  we propose a generalization of the concept of the Euler Square, namely, Generalized Euler Square (GES). %We show that GES of index $n,k,t$ - denoted GES($n,k,t$) - is \sout{which can be expressed}  an $n\times n^{t}$ rectangular array of $k$-ads 
%\sout{where there is}
%\sout{containing} possessing the properties: 
%(i). no intersection between any two $k-$tuples on the same column
%\sout{, there are},  
%(ii). at most $t-1$ intersections between any two distinct $k-$tuples of a same row  and (iii). at most $t$ intersections between any two distinct $k-$tuples which are not from the same row or column. It is clear that ES($n,k$) is nothing but GES($n,k,1$).
%\sout{In \cite{euler_1922}, the author has used results from group theory to construct Euler Squares .. (coming at wrong place and this sentence is repeated elsewhere .. hence not needed)}.
%\sout{After simplifying}. %For realizing our objective, we \sout{simplify the} look at the properties of Euler Squares with a different perspective \sout{, we able to} and then 
To begin with, we construct an Euler Square of index $p,k$, where $p$ is a prime or prime power, using polynomials of degree at most one over a finite field of order $p.$ 
%Motivated by this construction, we construct 
Generalized Euler Square of index $p,k,t$ are then constructed using polynomials of degree at most $t.$ 
%over a finite field of order $p.$ 
%Then we use 
Using a composition rule 
%to obtain 
a new GES for non-prime sizes are obtained from the combination of two GES of prime sizes. Further, we propose a methodology to construct compressed sensing matrices from GES. 
%which contains the matrices generated by Euler Square as a special case. 
The matrices 
%\sout{yields}
designed from GES show significant improvements in terms of column size.
In particular, the binary matrix constructed from GES($n,k,t$) has $n^{t}$ number of orthogonal blocks, each of size $nk\times n.$ We discuss the structure of the binary matrices generated from generalized Euler Squares, and show that they possess block orthogonality. We also derive that binary matrices constructed from Euler Square and GES possess low block coherence, providing thereby guarantees for recovering block sparse signals.
%\sout{We obtain a block orthogonal binary matrix by exploiting the block structure of Euler square to obtain a block orthogonal matrix. Then we show that the block orthogonality of the binary matrices along with low coherence provide guarantees for recovering block sparse signals. In contrast, the structure of Euler square was not utilized in construction presented in \cite{ram_2016} as the objective therein has been to construct binary matrix with small coherence for recovering sparse signals.}
%\sout{In order to obtain block orthogonal binary matrix with huge redundancy, we  define a generalization of Euler square, namely Generalized Euler Square (GES). The binary matrix coming from GES maintains block orthogonality as GES has block structure and huge redundancy comes from allowing more number of overlap between tuples in GES. We also show that GES are useful in recovering block sparse signals.
%Also, we present a simple construction of Euler Square using polynomials of degree at most one over a finite field.From there on, we construct GES using higher degree polynomials over a finite field. Then a composition is used to obtain new GES by combining two existing GES.}
The contributions of the present work may be summarized as follows:
\begin{itemize}

\item presenting Euler Squares with a different perspective and associating them  with polynomials of degree at most one over a finite field.

\item establishing block orthogonal structure of binary matrices constructed via Euler Squares in \cite{ram_2016} and showing their compliance with block sparse recovery properties

\item introducing Generalized Euler Squares  (GES) and providing their construction procedure
%\item constructing binary compressed sensing matrices, from GES, of better aspect ratio compared to their Euler Square based counterparts
\item bringing significant increment in column size compared to the ones obtained via Euler Squares
\item establishing block orthogonal structure of the GES based binary matrices.
\end{itemize}

%\textcolor{red}{(Like previous version, it will be good if contribution is summarized as bulleted points)}
 
%%%%%%%%%%%%%%%%%%%%%%%%%%%%%%%%%%%%%%%%%%%%%%%%%%%%%%%%%%%%%%%%%%%%%%%%%%%%%%%%%%%%%%%%%%%%%%%%%%%%%%%%%%%%%%%%%%%%%%%%%%%%%%%%%%%%%%%%%%%%%%%%%%%%%%%%%%%%%%%%%%%%%%%%%%%%%%%%%%%%%%%%%%%%%%%%%%%%%%%%%%%%%%%%%%%%%%%%%%%%%%%%%%%%%%%%%%%%%%%%%%%%%%%%%%%%%%%%%%%%%%%%%%%%%%%%%%%%%%%%%%%

\par The paper is organized in several sections. In section 2, we provide basics of CS theory and Euler Squares.  While in sections 3 and 4, we discuss  respectively the block sparse recovery through Euler Square based matrices and block orthogonality, and construction of Euler Square based matrices using polynomials of degree at most one over finite field. In sections 5 and 6, we introduce  Generalized Euler Squares and their construction respectively. In sections 7 and 8, we present respectively the construction of Euler Squares for composite order and construction of binary matrices using GES. In sections 9 and 10, we present GES as a rectangular array and the recovery guarantees for block sparse signals via GES. The paper ends with concluding remarks in Section 11. 
%and discuss the methodology to construct a CS matrix from a GES  respectively. 
%While in section 8, we discuss a simple way to construct ternary matrix from binary matrix, at the end we present our simulation result and concluding remarks. 

%%%%%%%%%%%%%%%%%%%%%%%%%%%%%%%%%%%%%%%%%%%%%%%%%%%%%%%%%%%%%%%%%%%%%%%%%%%%%%%%%%%%%%%%%%%%%%%%%%%%%%%%%%%%%%%%%%%%%%%%%%%%%%%%%%%%%%%%%%%%%%%%%%%%%%%%%%%%%%%%%%%%%%%%%%%%%%%%%%%%%%%%%%%%%%%%%%%%%%%%%%%%%%%%%%%%%%%%%%%%%%%%%%%%%%%%%%%%%%%%%%%%%%%%%%%%%%%%%%%%%%%%%%%%%%%%%%%%%%%%%%%

\section{Basics of Compressed Sensing and Euler Squares:}

\par The objective of compressed sensing is to recover $x = \{x_i\}^M_{i=1}\in \mathbb{R}^{M} $ from a few of its linear measurements $y\in \mathbb{R}^{m} $ through a stable and efficient reconstruction process via the concept of sparsity. From the measurement vector $y$ and the sensing mechanism, one  gets a system $y=\Phi x$, where $\Phi$ is an $m \times M \;(m < M)$ measurement matrix. %Denoting $\Phi^{r}_{k}$ as the $k^{\mbox{th}}$ row of $\Phi$, one may rewite the $k^{\mbox{th}}$ component in $y$ as $y_{k}=\left\langle x,\Phi^{r}_{k}\right\rangle, k=1,2,\dots m$. Here $\left\langle x,\Phi^{r}_{k}\right\rangle$ represents the inner-product between $x$ and $\Phi^{r}_{k}$. That is, the object $x$ to be acquired is correlated with the waveform $\Phi^{r}_{k}$. This is a standard setup in several applications \cite{Donoho_2006}\cite{candes_2008}. For example, if the sensing waveforms are Dirac delta functions, the $y$ is a vector of sampled values of $x$ in time or space domain. If the sensing waveforms are indicator functions of pixels, then $y$ is the image data typically collected by sensors in digital camera. If the sensing waveforms are sinusoids, then $y$ is a vector of Fourier coefficients and this modality is used in the magnetic resonance imaging. Nevertheless, if the sensing waveforms have $0$ and $1$ (or $0$ and $\pm1$) as elements, then the associated  matrix (referred conventionally to as a sensing matrix) can have potential application for multiplier-less dimensionality reduction. 
An excellent overview of Compressed Sensing and the applicability of several sensing matrices may be seen in \cite{candes_2008}.

Given the pair $(y, \Phi)$, the problem of recovering $x$ can be formulated as finding the sparsest solution (solution containing most number of zero entries) of linear system of equations $y=\Phi x$. Sparsity is measured by $\| . \|_{0}$ ``norm". $\|x\|_{0}$ denotes the number of non-zero entries in $x$, that is,  $\|x\|_{0} = |\{j: x_j \neq 0\}|$, where $\| . \|_{0}$ is neither a norm nor a quasi-norm. 
%\footnote{it violates the triangular inequality, that is $\| x_{1} + x_{2} \|_{0} \geq \| x_{1}\|_{0} + \|x_{2} \|_{0} $.}).
 Now finding the sparsest solution  can be formulated as the following minimization problem, generally denoted as $P_{0}$ problem:
\begin{displaymath}
P_0:\min_{x} \Vert{x} \Vert_0 \; \mbox{subject to} \quad \Phi x=y.
\end{displaymath} 
%This $P_{0}$ problem
This is a combinatorial optimization problem and is known to be NP-hard \cite{bourgain_2011}. One may use greedy methods and convex relaxation of $P_{0}$ problem to recover the $k-$sparse signals (i.e., signals for which $\|x\|_{0}\leq k$). The convex relaxation of $P_{0}$ problem can be posed as $P_{1}$ problem \cite{can_2008} \cite{Candes_2005}, which is defined as follows:
\begin{displaymath}
P_1:\min_{x} \Vert{x} \Vert_1 \; \mbox{subject to} \quad \Phi x=y.
\end{displaymath}
\par  Candes and Tao \cite{Tao_2006} have introduced the following isometry condition on matrices $\Phi$ and have established its important role in CS. An $m \times M$ matrix $\Phi$ is said to satisfy the Restricted Isometry Property(RIP) of order $k$ with constant $\delta_{k}$ $( 0<\delta_{k}<1)$ if for all vectors $x\in \mathbb{R}^{M}$ with $\|x\|_{0}\leq k$, we have
\begin{equation} 
\label{eqn:rip}
(1-\delta_{k}) \left\|x\right\|^{2}_{2} \leq \left\|\Phi x\right\|^{2}_{2} \leq (1+\delta_{k}) \left\|x\right\|^{2}_{2}.
\end{equation} 
Equivalently, for all vectors $x\in \mathbb{R}^{M}$ with $\left\|x\right\|_{2} = 1$ and $\|x\|_{0}\leq k$, one may rewrite (\ref{eqn:rip}) as
\begin{displaymath}
\label{eqn:r}
(1-\delta_{k})  \leq \left\|\Phi x\right\|^{2}_{2} \leq (1+\delta_{k}) .
\end{displaymath}
The following theorem \cite{Candes_2005} establishes the equivalence between $P_{0}$ and $P_{1}$ problems: 
\begin{theorem}
Suppose an $m \times M$ matrix $\Phi$ has the $(2k, \delta)$ restricted isometry property for some $\delta < \sqrt{2}-1$, then $P_{0}$ and $P_{1}$ have same $k-$sparse solution if $P_{0}$ has a $k-$sparse solution.
\end{theorem}
\par The coherence $\mu(\Phi)$ of a given matrix $\Phi$ is the largest absolute inner-product between different  normalized columns of $\Phi$, i.e., $\mu_{\Phi}= \max_{1\leq\; i,j \leq\; M,\; i\neq j} \frac{|\; \Phi_i ^T\phi_j|}{\Vert \Phi_i\Vert_{2} \Vert \Phi_j \Vert_2}$. Here, $\Phi_k$ stands for the $k$-th column in $\Phi$. 
\noindent The following proposition \cite{bourgain_2011} relates the RIP constant $\delta_{k}$ and $\mu_{\Phi}.$
\begin{proposition}
\label{thm:pro}
Suppose that $\phi_{1},\ldots,\phi_{M}$ are the unit norm columns of a matrix $\Phi$ with the 
coherence $\mu_{\Phi}.$ 
%Suppose $\phi_{i}$ the mutual coherence of $\Phi$ is $\mu_{\Phi}$. 
Then $\Phi$ satisfies RIP of order $k$ with constant $\delta_{k} = (k-1)\mu_{\Phi}$.
\end{proposition}
The Orthogonal Matching Pursuit (OMP) algorithm and and the $l_1-$norm minimization
(also called basis pursuit) are two widely studied CS reconstruction algorithms \cite{tropp_2010}. One of the important problems in CS theory deals with constructing CS matrices that satisfy the RIP for the largest possible range of $k$. It is known that the widest possible range of $k$ is of the order $\frac{m}{\log(\frac{M}{m})}$ \cite{Ronald_2007}, \cite{Richard_2008}, \cite{Kashin_1978}. However the only known matrices that satisfy the RIP for this range are based on random constructions \cite{Candes_2005}, \cite{Richard_2008}. %\textcolor{red}{(there is no paper in bibliography with nickname Richard\_2008)}

\subsection{Block sparse signal recovery}
Block sparse signals are sparse signals %with additional structure 
where the nonzero coefficients occur in clusters. A signal $x\in \mathbb{R}^{M}$ is viewed as a concatenation of $R$ number of blocks of length $d$ with $M=Rd.$ Denote the $\ell-$th block as $x[\ell],$ then
$$x=(x^{T}[1], x^{T}[2], \dots, x^{T}[R]),$$ where $x^{T}[\ell]=(x_{(\ell-1)d+1},x_{(\ell-1)d+2}, \dots, x_{\ell d}).$
The measurements can then be written as
$$y=\Phi x=\sum^{R}_{\ell=1}\Phi[\ell]x[\ell],$$ where $\Phi[\ell]=[\phi_{(\ell-1)d+1}, \phi_{(\ell-1)d+2}, \dots, \phi_{\ell d}]$ and $\phi_{j}$ is the $j-$th column of $\Phi.$
The vector $x$ is called block $k-$sparse if $x[\ell]$ has nonzero Euclidean norm for at most $k$ indices $\ell$. When $d=1$, block  sparsity reduces to conventional sparsity. Define $\|x\|_{2,0}$ as 
$$\|x\|_{2,0}=\sum^{R}_{\ell=1}I(\|x[\ell]\|_{2}>0),$$ where $I(\cdot)$ is the indicator function. A block $k-$sparse signal $x$ is defined as the signal that satisfies $\|x\|_{2,0}\leq k$ \cite{eldar_2010}, by definition.
\begin{definition}\cite{eldar_2010} Block coherence of a matrix $\Phi$ with normalized columns is defined as 
$$\mu_{B_{\Phi}}= \frac{1}{d}\max_{\ell, \ell \neq r}\lambda^{\frac{1}{2}}_{\max}(M^{T}[\ell,r]M[\ell,r]),$$ where $M[\ell,r]=\Phi^{T}[\ell]\Phi[r]$ and $\lambda_{\max}(A)$ is the largest eigen value of a positive semidefinite matrix $A.$
\end{definition}
It is known that $0\leq \mu_{B_{\Phi}}\leq \mu_{\Phi}$ \cite{eldar_2010}.
    
     \begin{proposition}
    \label{block_coherence}\cite{eldar_2010}
    If $\Phi$ consists of orthogonal blocks, that is, $\Phi^{T}[\ell]\Phi[\ell]=I_{d\times d}, \, \forall \, \ell, $ then $\mu_{B_{\Phi}}\leq \frac{1}{d}.$
    \end{proposition}

\begin{theorem}
\label{thm:bomp}
\cite{eldar_2010}
A sufficient condition for the Block OMP to recover a block $k-$sparse signal $x$, with each block length $d$, from $y=\Phi x$ is 
$$kd<\frac{\mu^{-1}_{B_{\Phi}}+d}{2},$$ where $\Phi$ is block orthogonal, that is, $\Phi^{T}[\ell]\Phi[\ell]=I_{d\times d}, \, \forall \, \ell.$
\end{theorem}

%%%%%%%%%%%%%%%%%%%%%%%%%%%%%%%%%%%%%%%%%%%%%%%%%%%%%%%%%%%%%%%%%%%%%%%%%%%%%%%%%%%%%%%%%%%%%%%%%%%%%%%%%%%%%%%%%%%%%%%%%%%%%%%%%%%%%%%%%%%%%%%%%%%%%%%%%%%%%%%%%%%%%%%%%%%%%%%%%%%%%%%%%%%%%%%%%%%%%%%%%%%%%%%%%%%%%%%%%%%%%%%%%%%%%%%%%%%%%%%%%%%%%%%%%%%%%%%%%%%%%%%%%%%%%%%%%%%%%%%%%%%

\subsection{Euler Square and  Generalized Latin square}
An Euler Square (ES) of order $n$, degree $k$ and index $n,k$ is a square array of $n^2$ $k-$ads of numbers, $(a_{ij1},a_{ij2},\ldots,
a_{ijk})$, where $a_{ijr} = 0,1,2,\ldots,n-1; r = 1,2,\ldots,k; i,j = 1,2,\ldots,n; n > k; a_{ipr} \neq a_{iqr}$ and $a_{pjr} \neq a_{qjr}$ for $p \neq q$ and $(a_{ijr},a_{ijs}) \neq (a_{pqr},a_{pqs})$ for $i \neq p$ and $j \neq q.$ In short we denote an Euler Square of order $n$, degree $k$ and index $n,k$ as ES($n,k$). Harris F. MacNeish \cite{euler_1922} has constructed Euler Squares for the following cases:
\begin{enumerate}
\item Index $p,p-1$, where $p$ is a prime number
\item Index $p^r, p^{r}-1$, for $p$ prime
\item Index $n,k$, where $n = 2^{r}p^{r_{1}}_{1}p^{r_{2}}_{2}\ldots, p^{r_{l}}_{l}$ for distinct odd primes $p_{1},p_{2},\ldots, p_l$. Here, $k = 
\min \{ 2^{r},p^{r_{1}}_{1},p^{r_{2}}_{2},\ldots ,p^{r_{l}}_{l} \} - 1.$
\end{enumerate}

\begin{lemma}\cite{euler_1922}\label{lem:e}
Let $k' < k$. Then the existence of the Euler Square of index $n,k$ implies that the Euler Square of index $n,k'$ exists.
\end{lemma}
%\textcolor{red}{(inserted Euler Square basics here .. earlier it was before beginning of Block sparse recovery for ES matrices .. its presence in betweek block sparse stuff looks to obstruct the flow)}
%{\color{green}{We refer to series of papers \cite{bose_1959,bose_1960,parker_1960, sade_1960} on the falsity of Euler's conjecture on the existence of Euler square.}}

%{\color{green}{
An Euler Square of degree one is called a Latin Square and of degree two a Graeco-Latin Square \cite{euler_1922,bose_1938}. Euler squares
are also called as mutually orthogonal Latin squares. 
In \cite{hedayat_1970, shen_1989}, the authors have proposed generalization of latin squares and their orthogonality.
%}}
%The conjecture of Euler that ES($n,2$) does not exist for $n \equiv 2 (\textrm{mod}\ 4)$ was disproved in a series of papers \cite{bose_1959,bose_1960,parker_1960}. 
% One of the most significant results in combinatorial
% theory is the disproof in 1959 and 1960 by Bose, Shrikhande and
% Parker in a succession of papers \cite{bose_1959,bose_1960,parker_1960} of the conjecture of Euler that ES($n,2$) does not exists for $n \equiv 2 (\textrm{mod}\ 4)$. 
%In fact, Euler’s conjecture turned out to be untrue except for $n$ = 2 or 6. In 1960, Sade \cite{sade_1960} provided some counter-examples to the Euler
%conjecture through singular direct product (SDP) of quasigroups.

\section{Block sparse signal recovery for  matrices constructed from Euler Square}
\label{ES_block_sparse}
%{\color{blue}{
The authors of \cite{ram_2016} have constructed binary matrix $\Phi^{n}_{k}$ of size $nk \times n^{2}$ with coherence at most $\frac{1}{k}$, whenever ES($n,k$) exists. In this section, it is shown that the binary matrices obtained from Euler Squares also contain a block orthogonal structure. 

%however, making use of the properties of Euler Square, we arrange the columns of  $\Phi^{n}_{k}$ and obtain a block orthogonal structure in it. To obtain the orthogonal block structure we look into the arrangement of $k-$tuples in Euler square. 

An Euler square can be observed to be an $n\times n$ square array of $k$-tuples with the following properties:
\begin{itemize}
	\item (ES 1) : each array entry is a $k-$tuple of numbers obtained from $\{0,\dots, n-1\}$
	\item (ES 2) : there is no intersection between any two $k-$tuples on the same row and same column
	\item (ES 3) : there is at most one intersection between any two distinct $k-$tuples which are not from the same
	row or column
\end{itemize}

Here intersection between two $k-$tuples denotes the number of indices where the corresponding entries of both the tuples are same.

\noindent For example, the Euler Square of index $3,2$ is stated as

$$0,0 \ \ \ \  1,1 \ \ \ \ 2,2 $$
$$1,2 \ \ \ \  2,0 \ \ \ \ 0,1 $$
$$2,1 \ \ \ \  0,2 \ \ \ \ 1,0 $$
\subsection{Block orthogonality}
%We make use of the property (ES 2) of Euler Square to show block orthogonality of binary matrices constructed from Euler Square.  

\begin{theorem}
\label{es_block_sparse}
If ES$(n,k)$ exists, then a sparse matrix of size $nk \times n^{2}$ exists which consists of $n$ orthonormal blocks, each of size $nk \times n.$  %\textcolor{red}{(size $2n \times n$ .. ?)}
\end{theorem}
\begin{proof}
From the construction procedure given in \cite{ram_2016}, a binary matrix $\Phi^{n}_{k}$ of size $nk \times n^{2}$ is obtained from ES($n,k$). Every column of $\Phi^{n}_{k}$ corresponds to a unique $k-$tuple of ES($n,k$). We arrange the columns of $\Phi^{n}_{k}$ to form a block orthogonal matrix. We form the $\ell-$th block (of block size $n$) by taking $n$ columns of $\Phi^{n}_{k}$ corresponding to $n$ $k-$tuples coming from $\ell-$th column of ES($n,k$). From the definition of Euler Squares, two distinct $k-$tuples belonging to same column of an Euler Square do not have any intersection. As a result, the inner product between any two distinct columns within a block is zero, implying thereby that each block is orthogonal. Consequently, the block matrix  $\frac{1}{\sqrt{k}}\Phi^{n}_{k}$ has $n$ orthonormal blocks, where each block is of size $nk\times n.$
\end{proof}
For example, let $\Phi^{3}_{2}$ be the binary matrix constructed from the ES($3,2$). Now the blocks of $\Phi^{3}_{2}$ corresponding to first, second and third columns are  

$\begin{pmatrix}
   1 & 0 & 0  \\
    0 & 1 & 0  \\
    0 & 0 & 1 \\
    1 & 0 & 0  \\
    0 & 0 & 1  \\
  0 & 1 & 0 
   \end{pmatrix}, $
$   \begin{pmatrix}
    0 & 0 & 1  \\
    1 & 0 & 0  \\
    0 & 1 & 0  \\
   0 & 1 & 0  \\
    1 & 0 & 0  \\
   0 & 0 & 1 
    \end{pmatrix}$ and $\begin{pmatrix}
    0 & 1 & 0 \\
    0 & 0 & 1 \\
    1 & 0 & 0 \\
    0 & 0 & 1 \\
    0 & 1 & 0 \\
    1 & 0 & 0 
    \end{pmatrix}$ respectively.

\noindent Now it is easy to see that the  blocks are orthogonal. Therefore $\Phi^{3}_{2}$ consists of $3$ orthogonal blocks, where each block is of size $6\times 3.$
%\begin{theorem}
%The block coherence of %$\frac{1}{\sqrt{k}}\Phi^{n}_{k_{r}}$ is at most $\frac{1}{n}.$
%\end{theorem}
%\begin{proof}
%The proof follows from Theorem~\ref{ges_block_sparse} and Proposition~\ref{block_coherence}.
%\end{proof}

%%%%%%%%%%%%%%%%%%%%%%%%%%%%%%%%%%%

\section{Construction of Euler Squares using finite fields}
\label{sec:new_construction}
In \cite{euler_1922}, the author has used group theoretical results to construct Euler Squares. 
%{\color{blue}{
In \cite{bose_1938}, R. C. Bose has constructed Euler square using polynomials over finite field. In this section, we revisit the construction given in \cite{bose_1938} and present a construction of Euler square using polynomials of degree at most one over finite fields. 
%}}

%Let us define the intersection between two $k-$ads of numbers as the count of indices where same numbers occur.  

%In simple term, if we take two distinct $k-$ads of numbers in an Euler Square of order $n$, degree $k$ and index $n,k$, same number can occur atmost at one position, that intersection between any two distinct $k-$ads of number is atmost $1.$  Therefore, an Euler Square of order $n$, degree $k$ and index $n,k$ is a square array of $n^2$ $k-$ads of numbers where intersection between any two distinct $k-$ads of numbers is atmost $1.$
\subsection{Construction procedure for prime or prime power order}
\begin{theorem}
Suppose $p$ is prime prime power. Then ES($p,p-1$) exists.
\end{theorem}

\begin{proof}
Let $D^{p}_{1}=\{P^{1}_{ij}=f_{i}x+f_{j}:i,j=1,2,\dots,p\}$ denote polynomials of degree at most one over a finite 
field $\mathbb{F}_{p}=\{f_{0}=0,f_{1}\dots,f_{p-1}\}$ of 
order $p.$ There are $p^{2}$ number of polynomials of degree at most one. Form the $k-$tuple, $S_{k}=
(f_{1},\dots,f_{k}) ,$ for $1\leq k\leq p-1.$ 
Evaluating a polynomial $P^{1}_{ij}$ of $D^{p}_{1}$ at every point of $S_{k},$ we form an ordered $k-$tuple 
$P^{1}_{ij}(S_{k})=(P^{1}_{ij}(f_{1}),\dots,P^{1}_{ij}(f_{k}))\in \mathbb{F}^{k}_{p}.$ Let us denote 
$S^{1}_{k}=\{P^{1}_{ij}(S_{k}): i,j=1,\dots,p\}\subseteq \mathbb{F}^{k}_{p}.$ It can be noted that
$|S^{1}_{k}|=p^{2}.$ To make $P^{1}_{ij}(S_{k_{p}})$ a $k-$tuple of numbers on $\{0,1,\cdots,p-1\},$ we replace 
$f_{i}$ with its index $i$.

We now claim that $S^{1}_{k}$ forms an Euler square of index $p,k$, where $i$ and $j$ denote row and column indices respectively. 
To show that $S^{1}_{k}=\{P^{1}_{ij}(S_{k})=(P^{1}_{ij}(f_{2}),\dots,P^{1}_{ij}(f_{k+1})\}: i,j=1,\dots,p\}$ 
forms an Euler Square of index $p,k,$ we need to show that, for $q,s=1,\dots,k$, $P^{1}_{in}(f_{q}) \neq 
P^{1}_{im}(f_{q})$ and $P^{1}_{nj}(f_{q}) \neq P^{1}_{mj}(f_{q})$ for $n \neq m$ and $(P^{1}_{ij}(f_{q}),P^{1}_{ij}
(f_{s})) \neq (P^{1}_{nm}(f_{q}),P^{1}_{nm}(f_{s}))$ for $i \neq n$ and $j \neq m.$

For $n \neq m,$ $P^{1}_{in}=f_{i}x+f_{n}$ and $P^{1}_{im}=f_{i}x+f_{m}$ do not have any common root, which shows that $P^{1}_{in}(f_{q}) \neq P^{1}_{im}(f_{q}).$ For $n \neq m,$ $P^{1}_{nj}=f_{n}x+f_{j}$ and $P^{1}_{mj}=f_{m}x+f_{j}$ have one common root at $f_{1}=0,$ which shows that $P^{1}_{nj}(f_{q}) \neq P^{1}_{mj}(f_{q}),$ as $1\neq q.$
For $i \neq n$ and $j \neq m,$ $P^{1}_{ij}$ and $P^{1}_{nm}$ can have at most one common root, which shows that $(P^{1}_{ij}(f_{q}), P^{1}_{ij}(f_{s})) \neq (P^{1}_{nm}(f_{q}),P^{1}_{nm}(f_{s})).$
Therefore, for prime or prime power $p$ and  $k\leq p-1$, one can construct $\text{ES}(p,k)$    using polynomials of degree at most one.
\end{proof}

%\begin{remark}
%The above construction procedure shows that we have $p-1 \choose k$ number of choices for $S_{k_{p}},$ and which implies that we have $p-1 \choose k$ copies of $\text{ES}(p,k).$

%\end{remark}

%For sake of simplicity of notation in the later part, let us denote $f_{i}=i$ to form an order among the elements of $\mathbb{F}_{p}.$
%
%Any two distinct polynomials of degree atmost one has atmost one common root. That shows that intersection between any two distinct $k-$tuples of $S^{1}_{k_{p}}$ is atmost one. This shows that $S^{1}_{k_{p}}$ forms an Euler Square of index $p,k.$

%When $p$ is prime we take $F_{p}$ as $\mathbb{Z}_{p},$ the field of integers of order $p$. When $p$ is prime power, then we take Galois field of order $p$ and denote the elements of the Galois field as $0,\dots, p-1.$
\subsection{Example in a prime or prime power case}
For constructing the Euler Square of index $3,2$, we consider the field $F_{3}=\mathbb{Z}_{3}=\{0,1,2\}.$ Then, the set 
$D^{1}_{3}=\{P^{1}_{ij}: i,j=0,1,2\}$ consists of all polynomials of degree at most one over $\mathbb{Z}_{3}.$ Note  that $|D^{1}_{3}|=9.$ Fix $S_{2}=(1,2)$ as the ordered $2-$tuple. Evaluating every polynomial of $D^{1}_{3}$  at every point of $S_{2},$ we get the set: \\ 
$S^{1}_{2}=\{(0,0),(1,2),(2,1); (1,1),(2,0),(0,2);(2,2),(0,1),  (1,0)\}\subseteq \mathbb{Z}^{2}_{3}.$ 
Now it is easy to check that $S^{1}_{2}$ forms an Euler square of index  $3,2.$ 
%by taking 1, 2 and 3 in places of 0, 1 and 2 respectively.

\subsection{Euler Square for composite order}
Once $\text{ES}(p,k)$ is obtained for $p$ being a  prime or prime power, one can follow the composition rule described in \cite{euler_1922} to obtain $\text{ES}(n,k)$ for composite 
$n = 2^{r}p^{r_{1}}_{1}p^{r_{2}}_{2}\ldots, p^{r_{l}}_{l}$  and  $k \leq \min \{2^{r}, p^{r_{1}}_{1}, p^{r_{2}}_{2}, \ldots, p^{r_{l}}_{l}\} - 1$, where $p_{1},p_{2},\ldots, p_l$ are 
distinct odd primes. 
%{\bf Pradip: In previous sentence check if $k-1 < \min\{ ..\}$ or $k < \min\{ .. \} $}

\section{Generalized Euler Square}
In this section, we propose a generalization of Euler Squares.
%, namely, Generalized Euler Square (GES).
\begin{definition}{Generalized Euler Square (GES)}: \\
	A GES of index $n,k,t$, denoted GES($n,k,t$),  with $n>k>t$, is a hyper-rectangle of $n^{t+1}$ $k-$ads 
	of numbers, $(a_{i_{1}i_{2}\dots i_{t+1}1}, \dots , a_{i_{1}i_{2}\dots i_{t+1}k} )$, 
	%\footnote{$(a_{i_{1}i_{2}\dots i_{t+1},1}, \dots , a_{i_{1}i_{2}\dots i_{t+1},k} )$ .. I guess this is correct with commas before 1 and k ..}, 
	where $a_{i_{1}i_{2}\dots 
		i_{t+1}r} \in \{0,1,\dots , n-1\};$ $ r=1,2,\dots, k;$ $1\leq i_{s}\leq n$ for $s=1,\cdots, t+1;$   $a_{i_{1}\dots 
		i_{j-1}ui_{j+1}\dots i_{t+1}r}\neq a_{i_{1}\dots i_{j-1}vi_{j+1}\dots i_{t+1}r}$ for $u\neq v$ and for 
	$i_{x_{1}}\neq j_{x_{1}}, i_{x_{2}}\neq j_{x_{2}}; 1 \leq x_{1} \neq x_{2} \leq t+1,$ $(a_{i_{1}i_{2}\dots 
		i_{t+1}r_{1}}, \dots ,a_{i_{1}i_{2}\dots i_{t+1}r_{t+1}}) \neq (a_{j_{1}j_{2}\dots j_{t+1}r_{1}}, \dots 
	,a_{j_{1}j_{2}\dots j_{t+1}r_{t+1}})$, where $1 \leq r_l \leq k$, $1 \leq l \leq k+1$.
\end{definition}

\begin{remark}
	It is easy to check that an Euler Square of index $n,k$ is a GES of index $n,k,1.$
\end{remark} 

\noindent 
%\sout{We now provide the construction of GES.} 
In line with Harris F. MacNeish's construction \cite{euler_1922} for Euler Square, we construct the Generalized Euler Squares for the following cases:
	\begin{itemize}
		\item Index $p,p-1,t$, where $p$ is a prime number
		\item Index $p^r, p^{r}-1,t$, for $p$ prime
		\item Index $n,k,t$, where $n = 2^{r}p^{r_{1}}_{1}p^{r_{2}}_{2}\ldots, p^{r_{l}}_{l}$ for distinct odd primes $p_{1},p_{2},\ldots, p_l$. Here, $k + 1$ equals the least of the numbers $2^{r},p^{r_{1}}_{1},p^{r_{2}}_{2},\ldots ,p^{r_{l}}_{l}.$
	\end{itemize}
%Before going into the construction of GES, we provide a construction of Euler Square using polynomials of degree at most one over a finite field. 
In the next section, we use higher degree polynomials over finite field for constructing GES.
%}}

\section{Construction of Generalized Euler Squares}
\label{sec:GES}
%{\color{blue}{
One of the main objectives of proposing Generalized Euler Square is to obtain a binary matrix possessing a larger number of columns.
%\sout{. Consequently, we need to}, which is possible when we have more number of $k-$tuples in GES.
However, this comes at the cost of increased
%But the cost of having more number $k-$tuples comes with more 
number of intersections between 
%two distinct 
the $k-$tuples. Therefore, in GES we allow more than one intersections in order to produce large number of $k-$tuples.
%}}
The intersection between any two distinct $k-$tuples is the number of common roots between the two corresponding polynomials.
%that are used to construct initial two $k-$tuples. 
%\sout{Now if we allow at most $r$ intersections between two $k-$tuples, we can use polynomials of degree at most $r$ over the finite field $\mathbb{F}_{p},$ as the number of common roots between any two distinct polynomials of degree at most $r$ is at most $r$}. 
The use of polynomials of higher degree allows for more intersections between the $k-$tuples. 
%({\bf don't know why these sentences were blocked in tex file}). 
% \sout{Extending Euler squares to higher dimensions} 
%Motivated by this observation, we use higher degree polynomials as opposed to the polynomials of degree one and obtain the Generalized Euler squares. \\

%Note that each $i_j$ is allowed to vary from $1,\cdots,n$ for each $1 \leq j \leq t+1$.
%Each $1 \leq i \leq t+1$ is considered as a dimension of the Generalized Euler square. 
%The two properties above can be 
%described as: 
%1) No two $k$-ads lying \sout{in the same $t$ dimensions} %\textcolor{red}{along the same coordinate or dimension (.. ?) } can intersect and
%2) Any two $k$-ads lying in at least two different \textcolor{red}{coordinates} \sout{dimensions} will have any $t+1$ of their subsets to be unequal. \textcolor{red}{(graphical or pictorial explanation will help understand it easily)}

%\textcolor{red}{Gen. Latin Square is available in literature. How is present GES different from those available in literature .. Is it possible to write a short subsection on this comparison ?}
\subsection{Construction of GES of index $p,k,t$ where $p$ is a prime or prime power}

Consider the polynomials of degree at most $t$ over a finite field $\mathbb{F}_{p}=\{f_{0}=0,f_{1},\dots,f_{p-1}\}$ of order 
$p.$ Form an ordered $k-$tuple $S_k=(f_{1},....,f_{k}).$ Let us denote the set of 
polynomials of degree at most $t$ as $D_{t}.$ 
As there are $p^{t+1}$ number of polynomials of degree at most 
$t$, we can write $D_{t}=\{P^{t}_{i_{1}i_{2}\dots i_{t+1}}=\sum^{t+1}_{j=1}f_{i_{j}}x^{j-1}:f_{i_{j}}\in 
\mathbb{F}_{p},1 \leq i_{j}\leq p\}.$
%\footnote{$i_{j}=1,\dots,p$ OR $1 \leq i_{j}\leq p$, which one is precise ... ?}$.  
Evaluating a polynomial $P^{t}_{i_{1}i_{2}\dots i_{r+1}}$ at 
every point of $S_{k},$ we form an ordered $k-$tuple $P^{t}_{i_{1}i_{2}\dots i_{t+1}}(S_{k})=
(P^{t}_{i_{1}i_{2}\dots i_{t+1}}(f_{1}),\dots,P^{t}_{i_{1}i_{2}\dots i_{t+1}}(f_{k}))\in \mathbb{F}^{k}_{p}.$ Let 
$S^{t}_{k}=\{P^{t}_{i_{1}i_{2}\dots i_{t+1}}(S_{k}):P^{t}_{i_{1}i_{2}\dots i_{t+1}}\in 
D^{p}_{t}\}\subseteq \mathbb{F}^{k}_{p}.$ Now $|S^{t}_{k}|=p^{t+1}.$ Similar to the 
Euler Square case, in order to make $P^{t}_{i_{1}i_{2}\dots i_{r+1}}({S_{k}})$ a $k-$tuple of numbers on 
$\{0,\cdots,p-1\},$ we replace $f_{i}$ with its index
$i$. It will be shown next that the $k$-tuples in the set $S^{t}_{k}$ form a GES($n,k,t$).

For $u\neq v,$ $P^{t}_{i_{1}\dots i_{j-1} u i_{j+1} \dots i_{t+1}}$ and
$P^{t}_{i_{1}\dots i_{j-1} v i_{j+1} \dots i_{t+1}}$ can have either no common root
when $j=1$ or $0 \in \mathbb{F}_p$ as its common root 
when $1 < j \leq t+1$. This shows 
that, $P^{t}_{i_{1}\dots i_{j-1} u i_{j+1} \dots i_{t+1}}(S_{k})$ and
$P^{t}_{i_{1}\dots i_{j-1} v i_{j+1} \dots i_{t+1}}(S_{k})$ have no intersection as
$0\notin {S_{k}}$. For $i_{x_{1}}\neq j_{x_{1}}, i_{x_{2}}\neq j_{x_{2}}; 1\leq x_{1}\neq x_{2} \leq t+1$, 
$P^{t}_{i_{1}i_{2}\dots i_{t+1}}$ and $P^{t}_{j_{1}j_{2}\dots j_{t+1}}$ have at most $t$
number of common roots. This proves that $S^{t}_k$ forms GES($n,k,t$).

%\begin{remark}
%The above construction procedure shows that we have $p-1 \choose k$ number of choices for $S_{k_{p}}$ and that means we have $p-1 \choose k$ isomorphic copies of GES($n,k,t$)
%
%\end{remark}

%{\tiny{
%\begin{bmatrix}
%0 & x & 2x & 3x & 4x &  \dots & 4x^{2 } & 4x^{2}+ x & 4x^{2}+2x & 4x^{2}+3x & 4x^{2}+4x \\
%1 & x+1 & 2x+1 & 3x+1 & 4x+1 & \dots & 4x^{2 } & 4x^{2}+ x+1 & 4x^{2}+2x+1 & 4x^{2}+3x+1 & 4x^{2}+4x+1\\
%2 & x & 2x+2 & 3x+2 & 4x+2 &  \dots & 4x^{2 }+2 & 4x^{2}+ x +2& 4x^{2}+2x +2& 4x^{2}+3x +2& 4x^{2}+4x+2\\
%3 & x & 2x+3 & 3x+3 & 4x+3 & \dots & 4x^{2 } +3& 4x^{2}+ x +3& 4x^{2}+2x+3 & 4x^{2}+3x+3 & 4x^{2}+4x+3 \\
%4 & x+4 & 2x+4 & 3x+4 & 4x+4 &  \dots & 4x^{2 }+4 & 4x^{2}+ x+4 & 4x^{2}+2x+4 & 4x^{2}+3x+4 & 4x^{2}+4x+4 \\
%\end{bmatrix}
%}}

%Now evaluating the polynomials at $\mathbb{Z}^{4}_{5}$ we get the following rectangle of size $5\times 25$ with each entry being a $4-$tuple.

%\begin{bmatrix}
%(0,0) & (1,2) & (2,1) & (1,1) &  x^{2}+x & x^{2}+2x & x^{2} + 2x & 2x^{2 } & 2x^{2}+ x & 2x^{2}+2x \\
%1 & x + 1 & 2x + 1 & x^{2}+1 &  x^{2}+x+1 & x^{2}+2x+1 & x^{2} + 2x +1 & 2x^{2 }+ 1 & 2x^{2}+ x + 1 & 2x^{2}+2x + 1\\
%2 & x + 2 & 2x + 2 & x^{2} +2 &  x^{2}+x+2 & x^{2}+2x+2 & x^{2} + 2x +2 & 2x^{2 } + 2 & 2x^{2}+ x + 2 & 2x^{2}+2x + 2
%\end{bmatrix}

\section{Construction of GES for composite order}
In this section, we construct GES for composite dimensions. 
We now describe a procedure to combine two GES to produce a new GES. The resulting GES 
can have more general row sizes which are different from prime or prime powers. 

Let $p'$ and $p''$ be two primes or prime powers. Then, with $1 \leq t<k<p''\leq p',$ we can obtain 
one GES of index $p',k,t$ and another  GES of index $p'',k,t$ from the previous construction. 
Let $\{(c'_{i_{1},i_{2}, \ldots ,i_{t+1},1}, c'_{i_{1},i_{2}, \ldots ,i_{t+1},2}\ldots,c'_{i_{1},i_{2}, \ldots ,i_{t+1},k})\}^{p'}_{i_{s}=1}$ denote the GES$(p',k,t)$ and 
$\{(c''_{j_{1},j_{2}, \ldots ,j_{t+1},1}, c''_{j_{1},j_{2}, \ldots ,j_{t+1},2}\ldots,c''_{j_{1},j_{2}, \ldots ,j_{t+1},k})\}^{p''}_{j_{s}=1}$ denote GES$(p'',k,t).$ Let us define that
$$c_{i_{1}+p'(j_{1}-1),i_{2}+p'(j_{2}-1), \ldots ,i_{t+1}+p'(j_{t+1}-1),r}=c'_{i_{1},i_{2}, \ldots 
	,i_{t+1},r}+p'(c''_{j_{1},j_{2}, \ldots ,j_{t+1},r}-1).$$ Take $m_{s}=i_{s}+p'(j_{s}-1).$ It is clear that 
\begin{enumerate}
	\item $1\leq 
	m_{s}\leq p'p''$ as $1\leq i_{s}\leq p'$ and $1\leq j_{s}\leq p''$ 
	
	\item $0\leq c_{m_{1},m_{2}, \ldots 
		,m_{t+1},r}\leq p'p''-1$ as $0\leq c'_{i_{1},i_{2}, \ldots ,i_{t+1},r}\leq p'-1$ and $0\leq c''_{j_{1},j_{2}, \ldots 
		,j_{t+1},r}\leq p''-1.$ 
\end{enumerate}
It is now shown that the $k$-ads consisting of elements
$c_{m_{1},m_{2}, \ldots ,m_{t+1},r}$ for $r = 1,2,\cdots, k$ for $m_{i}=1,2,\cdots,p^\prime p^{\prime\prime}$,  
$i = 1,2,\cdots,t+1$ form GES$(p'p'',k,t)$.
\begin{enumerate}
	\item 
	$m_s^a \neq m_s^b$ implies that $i_s^a+p'(j_s^a-1)\neq i_s^b+p'(j_s^b-1)$. This 
	can happen for 
	$i_s^a\neq i_s^b$
	or $j_s^a\neq j_s^b$ or both. 
	By definition, for $i_s^a\neq i_s^b,$ $c'_{i_{1},i_{2}, \ldots ,i_{s-1},i_s^a,i_{s+1},\ldots, 
		i_{t+1},r}\neq c'_{i_{1},i_{2}, \ldots ,i_{s-1},i_s^b,i_{s+1},\ldots, i_{t+1},r}$ and for $j_s^a\neq 
	j_s^b,$ $c''_{j_{1},j_{2}, \ldots ,j_{s-1},j_s^a,j_{s+1},\ldots, j_{t+1},r}\neq c''_{j_{1},j_{2}, \ldots 
		,j_{s-1},j_s^b,i_{j+1},\ldots, j_{t+1},r}$. Also note that $|c'_{i_{1},i_{2}, \ldots ,i_{s-1},i_s^a,i_{s+1},\ldots, 
		i_{t+1},r} - c'_{i_{1},i_{2}, \ldots ,i_{s-1},i_s^b,i_{s+1},\ldots, i_{t+1},r}| < p'$. Now, if we have $m_s^a\neq m_s^b,$  $c_{m_{1},m_{2}, 
		\ldots ,m_{s-1},m_s^a,m_{s+1},\ldots, m_{t+1},r}\neq c_{m_{1},m_{2}, \ldots ,m_{s-1},m_s^b,m_{s+1},\ldots, 
		m_{t+1},r}.$ 
	\item Suppose that $m_s^a\neq m_s^b$ and $m_{\tilde{s}}^a\neq m_{\tilde{s}}^b.$ Now $m_s^a\neq 
	m_s^b$ can happen for $i_s^a\neq 
	i_s^b$ or $j_s^a\neq j_s^b$ or for both whereas $m_{\tilde{s}}^a\neq m_{\tilde{s}}^b$ can occur if 
	$i_{\tilde{s}}^a\neq i_{\tilde{s}}^b$  or 
	$j_{\tilde{s}}^a\neq j_{\tilde{s}}^b$ or both. Hence $m_s^a\neq m_s^b$ and $m_{\tilde{s}}^a\neq 
	m_{\tilde{s}}^b$ can happen for nine possible pairs of combinations which can be formed by taking one case from 
	occurrence of $m_s^a\neq m_s^b$ and another case from the occurrence of $m_{\tilde{s}}^a\neq 
	m_{\tilde{s}}^b.$  For $i_s^a\neq i_s^b$ or $i_{\tilde{s}}^a\neq 
	i_{\tilde{s}}^b$ or both,  
	\begin{equation*}
	\begin{split}
	& (c'_{i_{1}^a,\ldots,i_s^a-1,i_s^a,i_s^a+1,\ldots, 
		i_{\tilde{s}}^a-1,i_{\tilde{s}}^a,i_{\tilde{s}}^a+1,\ldots,i_{t+1}^a,r_{z}})^{t+1}_{z=1}   \\ 
	& \quad \quad \quad \quad \quad \quad \quad \quad  \neq (c'_{i_{1}^b,\ldots,i_s^b-1,i_s^b,i_s^b+1,\ldots, 
		i_{\tilde{s}}^b-1,i_{\tilde{s}}^b,i_{\tilde{s}}^b+1,\ldots,i_{t+1}^b,r_{z}})^{t+1}_{z=1}.
	\end{split}
	\end{equation*}
	
	Similarly for $j_s^a\neq j_s^b$ or $j_{\tilde{s}}^a\neq j_{\tilde{s}}^b$ or both, 
	\begin{equation*}
	\begin{split}
	& (c''_{j_{1}^a,\ldots,j_s^a-1,j_s^a,j_s^a+1,\ldots, 
		j_{\tilde{s}}^b-1,j_{\tilde{s}}^b,j_{\tilde{s}}^b+1,\ldots,j_{t+1}^a,r_{z}})^{t+1}_{z=1} \\
	& \quad \quad \quad \quad \quad \quad \quad \quad \neq (c''_{j_{1}^b,\ldots,j_s^b-1,j_s^b,j_s^b+1,\ldots, j_{\tilde{s}}^b-1,j_{\tilde{s}}^b,j_{\tilde{s}}^b+1,\ldots,j_{t+1}^b,r_{z}})^{t+1}_{z=1}.
	\end{split}
	\end{equation*}
	Hence,
	\begin{equation*}
	\begin{split}
	& (c_{m_{1}^a,\ldots,m_s^a-1,m_s^a,m_s^a+1,\ldots, 
		m_{\tilde{s}}^a-1,m_{\tilde{s}}^a,m_{\tilde{s}}^a+1,\ldots,m_{t+1}^a,r_{z}})^{t+1}_{z=1} \\
	& \quad \quad \quad \quad \quad \quad \quad \quad \neq  (c_{m_{1}^b,\ldots,m_s^b-1,m_s^b,m_s^b+1,\ldots, 
		m_{\tilde{s}}^b-1,m_{\tilde{s}}^b,m_{\tilde{s}}^b+1,\ldots,m_{t+1}^b,r_{z}})^{t+1}_{z=1}.
	\end{split}
	\end{equation*}
\end{enumerate}

\subsection{GES for general order}
\noindent A natural extension of the above composition can be stated as follows:
\begin{theorem}
\label{property_ges}
	Suppose $n = 2^{r}p^{r_{1}}_{1}p^{r_{2}}_{2}\ldots, p^{r_{l}}_{l}$ for distinct odd primes $p_{1},p_{2},\ldots, p_l$ 
	and $t<k<\min\{2^{r},p^{r_{1}}_{1},p^{r_{2}}_{2},\ldots ,p^{r_{l}}_{l}\},$ then GES($n,k,t$) exists. 
	%and it can be represented as an $n \times n^{t}$ matrix with each entry being a $k-$tuple of number taken from $\{1,2,\dots,n\}$ and has the following properties:
	%\begin{itemize}
%	\item Two distinct $k-$tuples from the same column do not intersect.
%	\item Two distinct $k-$tuples from same row have at most $t-1$ intersections.
%	\item Any two distinct $k-$tuples in the array can have at most $t$ intersections.
%\end{itemize}

\end{theorem}

\noindent Similar to the Euler Square case (Lemma \ref{lem:e}), we have the following result:
\begin{lemma}
	\label{lem:gen}
	Let $t< k' < k$. Then the existence of GES($n,k,t$) implies that GES($n,k',t$) exists.
\end{lemma}

\section{Construction of binary matrices via GES}
In this section, a construction of CS matrices from GES($n,k,t$) is proposed. Let us represent $n^{t+1}$ number of $k-$tuples as $\{(t^{j}_{1}, \ldots, t^{j}_{k}): j=1,\ldots,n^{r+1}\},$ obtained from GES($n,k,t$). Note that, $0 \leq t^{j}_{i}\leq n-1,$ for all $j=1,\ldots,n^{r+1}$ and $i=1,\dots, k.$  
From a $k-$tuple $(t^{j}_{1}, \ldots, t^{j}_{k}),$ we form a binary vector $ v^{j}$ of length $nk$ such that
$$v^{j}(i)=\begin{cases} 
1, & \text{if} \, i= (l-1)n+t^{j}_{l}+1, 1\leq l \leq k \\
0, &  \text{elsewhere}
\end{cases}$$
 Using $n^{t+1}$ number of $k-$tuples, $n^{t+1}$  binary vectors, each of length $nk$, are obtained. By taking the binary vectors as columns, we obtain a binary matrix $\Phi(n,k,t)$ of size $nk \times n^{t+1}.$ We treat  $\Phi(n,k,t)$ as the GES matrix of index $n,k,t.$ 

\subsection{Properties of GES matrix}
%Let us discuss the properties of GES matrix $\Phi(n,k,r)$ of index $n,k,r.$
\begin{enumerate}

\item %By construction process, 
$\Phi(n,k,t)$ has $k$ number of blocks with each block size being $n.$ 
\item Each column of $\Phi(n,k,t)$ has exactly $k$ number of ones and contains a single $1$ in each block.
\item As intersection between any two distinct $k-$tuples of GES($n,k,t$) is at most $t,$ the non-zero overlap between any two distinct columns of $\Phi(n,k,t)$ is at most $t.$
\end{enumerate}

\begin{lemma}
\label{thm:lem}
The coherence $\mu_{\Phi(n,k,r)}$ of $\Phi(n,k,r)$ is at most equal to $\frac{r}{k}$.
\end{lemma} 
\begin{proof}
Each column of $\Phi(n,k,r)$ has exactly $k$ number of ones, which implies $\ell_{2}-$norm of each column is $\sqrt{k}.$ Also the non-zero overlap between any two distinct columns of $\Phi(n,k,r)$ is at most $r.$ So, the absolute value of the inner product between any two distinct columns is at most $r.$ Hence the coherence $\mu_{\Phi(n,k,r)}$ of $\Phi(n,k,r)$ is at most $\frac{r}{k}.$
\end{proof}
%%%%%%%%%%%%%%%%%%%%%%%%%%%%%%%%%%%%%%%%%%%%%%%%%%%%%%%%%%%%%%%%%%%%%%%%%%%%%%%%%%%%%%%%%%%%%%%%%%%%%%%%%%%%%%%%%%%%%%%%%%%%%%%%%%%%%%%%%%%%%
%\begin{remark}
%If the row size ($m$) and column size ($M$) of $\Phi^{n}_{k_{r}}$  are $nk$ and $n^{r+1}$ respectively, then we obtain the following relation among $m,M$ and $\mu_{\Phi^{n}_{k_{r}}}:$ 
 %$$\mu_{\Phi^{n}_{k_{r}}}=\frac{r}{k}=\frac{rM^{\frac{1}{r+1}}}{m}.$$
%\end{remark}

\begin{remark} 
The maximum possible column size of any binary matrix is $\frac{{m \choose t+1}}{{k \choose t+1}}$ \cite{amini_2011}, where $m$ is the row size, $k$ is the number of ones in each column and $t$ is the maximum overlap between any two columns. If $m=nk,$ which is the case for $\Phi(n,k,t),$ and $t$ is fixed, then the maximum possible column size is 

\begin{equation}\label{maxcol}
\frac{{nk \choose t+1}}{{k \choose t+1}} = \Theta(n^{t+1}), 
%= \Theta \big((\frac{m\mu_{\Phi^{n}_{k_{r}}}}{r})^{r+1}\big),
\end{equation}
where $a= \Theta(b)$ implies that, there exist two constants $c_{1}, c_{2}$ such that $c_{1}b \leq a \leq c_{2}b$.
Hence the column size of $\Phi(n,k,t)$ is in the maximum possible order.
%For the GES matrix $\Phi^{n}_{k_{r}}, m= nk, M= n^{r+1}$ and $\mu_{\Phi^{n}_{k_{r}}} = \frac{r}{k}$, and hence $M= (\frac{m\mu_{\Phi^{n}_{k_{r}}}}{r})^{r+1}$, which is in the order of maximum possible column size. Hence the aspect ratio is also in the maximum possible order. 
\end{remark}
%which satisfies $n^2 \leq \frac{{nk \choose 2}}{{k \choose 2}} \leq 2n^2$ and is asymptotically equivalent to $n^2$ as $n, k \rightarrow \infty$. 
%Therefore in this construction our column size asymptotically reaches the maximum possible column size.
%The density for GES matrices is $\frac{1}{n}$. The following proposition \cite{bourgain_2011} relates the RIP constant $\delta_{k'}$ and $\mu.$
%\begin{proposition}
%\label{thm:pro}
%Suppose that $\phi_{1},\ldots,\phi_{m}$ are the unit norm columns of the matrix $\Phi$ possessing the coherence $\mu$. Then $\Phi$ satisfies the RIP of order $k'$ with constant $\delta_{k'} = (k'-1)\mu$. 
%\end{proposition}
\noindent From lemma \ref{thm:lem} and Proposition \ref{thm:pro}, it follows that the matrix $\Phi(n,k,t)$ so constructed satisfies RIP.
\begin{theorem}
\label{thm:main}
The matrix $\Phi_{0} = \frac{1}{\sqrt{k}}\Phi(n,k,t)$ satisfies RIP with $\delta_{k'} = \frac{t(k'-1)}{k}$ for any $k' < \frac{k}{t} + 1$. 
\end{theorem}
%%%%%%%%%%%%%%%%%%%%%%%%%%%%%%%%%%%%%%%%%%%%%%%%%%%%%%%%%%%%%%%%%%%%%%%%%%%%%%%%%%%%%%%%%%%%%%%%%%%%%%%%%%%%%%%%%%%%%%%%%%%%%%%%%%%%%%%%%%%%%
%
%\begin{theorem}
%Suppose $m$ is any positive non prime integer. Then there exists a CS matrix of size $m\times (\frac{m\mu}{r})^{r+1}$, where $\mu$ is the coherence of the matrix and $r$ is smaller than the smallest prime power factor of $m$.
 
%If $m=p^{r_{1}}_{1}p^{r_{2}}_{2}\ldots p^{r_{l}}_{l}$ such that $r_{i}\geq 2$ for $1 \leq i \leq l$, $r_{i} > 2$ for $i=1$ then there exist a matrix of row size $m$.
%\end{theorem} 
%\begin{proof}
%Suppose for the $m$ is a positive non-prime integer. Let $m =p_{1}p_{2}\ldots p_{l}$ for primes or prime powers $p_{1},p_{2},\ldots, p_{l}$ with $p_{1}\leq p_{2}\leq \ldots \leq p_{l}$ and take $1\leq r<k \leq p_{1}.$ Let $n=p_{2}\ldots p_{l}$ Then there exists Generalized Euler Square(GES) of index $n,k,r.$ Let us consider GES matrix $\Phi^{n}_{p_{1_{r}}}$ of index $n, p_{1}, r.$ Here the size of $\Phi^{n}_{p_{1_{r}}}$ is $m\times n^{r+1}$ and coherence $\mu_{\Phi^{n}_{p_{1_{r}}}}$ is $\frac{r}{p_{1}}.$ This completes the proof of the theorem.

%\end{proof}
%%%%%%%%%%%%%%%%%%%%%%%%%%%%%%%%%%%%%%%%%%%%%%%%%%%%%%%%%%%%%%%%%%%%%%%%%%%%%%%%%%%%%%%%%%%%%%%%%%%%%%%%%%%%%%%%%%%%%%%%%%%%%%%%%%%%%%%%%%%%%

%\begin{remark}
%In \cite{ram_2016}, authors have constructed CS matrix from Euler Square of index $n,k,$ which is nothing but the GES matrix $\Phi^{n}_{k_{1}}$ of index $n,k,1.$ 
%\end{remark}

\section{GES as rectangular array}

%{\color{blue}{
%Like Euler Square, the Generalized Euler Square (GES) can be viewed as an $n\times n^{t}$ rectangle 
\begin{definition}(Rectangular Array:) A rectangular array is a two dimensional
  array of $k$-tuples with the following properties:
\begin{itemize}
\item (GES 1) : each array entry is a $k-$tuple of numbers obtained from $\{0,\dots, n-1\}$, 
	\item (GES 2) : two distinct $k-$tuples from the same column do not intersect, 
	\item (GES 3) : two distinct $k-$tuples from same row have at most $t-1$ intersections,
	\item (GES 4) : any two distinct $k-$tuples in the array can have at most $t$ intersections.
\end{itemize}  
\end{definition}

\begin{theorem}
\label{es_square_prime}
For $p$ being a prime or prime power, ES($p,p-1$) is a $p\times p$ square matrix 
%with each entry being 
of $(p-1)-$tuple such that 
\begin{itemize}
	\item each entry is a $(p-1)-$tuple of numbers obtained from $\{0,\dots, p-1\}$,
	\item there is no intersection between any two $(p-1)-$tuples on the same row and same column,
	\item there is exactly one intersection between any two distinct $(p-1)-$tuples which are not from the same row or column.
\end{itemize}
\end{theorem}

\begin{proof}
%This a constructive proof.

Consider the finite field
$\mathbb{F}_{p}=\{f_{0}=0,f_{1},\dots,f_{p-1}\}$
where $p$ is a prime or a prime power.
Let $S^{p}$ be the collection of polynomials of degree at most
$1$, with  zero being the constant term. It is easy to
check that the cardinality
of $S^{p}$ is $|S^{p}|=p.$  For $P\in S^{p},$ define the set
$S^{p}_{P}=\{P_{j}=P+f_{j}:j=0, \dots, p-1\}$. Fix any ordered $(p-1)-$tuple
$z\in \mathbb{F}_p^(p-1)$ with $k = p-1$. For simplicity, we consider
$z=(f_{1}, \dots, f_{p-1})$.
An ordered $(p-1)-$tuple is formed
after evaluating $P_{j}$ at each of the points of $z$, that is, $d^{P}_{j} := \big(P_{j}(f_{1}), \cdots, P_{j}(f_{p-1})\big)$.
In order to make $d^{P}_{j}$ a $(p-1)-$tuple of numbers on 
$\{0,\cdots,p-1\},$ we replace $f_{i}$ with its index 
$i$.  Now $d^{P}_{j}$ for $j=0, \dots,p-1$ forms one column and as $|S^{p}|=p,$ there are $p$ such columns. Therefore we get a matrix of  size $p\times p$ with each entry being $(p-1)-$tuples. 

Let us take two $(p-1)-$tuples from same column. As they belong to same column the corresponding polynomials are of the form $P+f_{i}$ and $P+f_{j}$ for $i\neq j$ and $P$ being a polynomial of degree at most $1$, with  zero being the constant term. Since  $P+f_{i}$ and $P+f_{j}$ does not share any common root, there is no intersection between the tuples coming from same column. 

Let us take two $(p-1)-$tuples from same row. As they belong to same row the corresponding polynomials are of the form $P_{1}+f_{i}$ and $P_{2}+f_{i}$ for $P_{1}$ and $P_{2}$ being a polynomial of degree at most $1$ with  zero being the constant term. Since  $P_{1}+f_{i}$ and $P_{2}+f_{i}$ share zero as their only common root, there is no intersection between the tuples coming from same row as the polynomials are not evaluated at zero while forming the tuples.
%Now, it is easy to  check first that there is no intersection between any two distinct $k-$tuples coming from the same row and the same column. 
Let us take two $(p-1)-$tuples from different row and column. As they belong to different row and different column  the corresponding polynomials are of the form $P_{1}+f_{i}$ and $P_{2}+f_{j}$ for $i\neq j$ with $P_{1}$ and $P_{2}$ being polynomials of degree at most $1$ and  zero being the constant term. Since  $P_{1}+f_{i}$ and $P_{1}+f_{j}$ share exactly one nonzero common root, there is exactly one intersection between the tuples coming from different row and different column.

\end{proof}

\begin{example}

Let us first take $\mathbb{F}_{3}=\mathbb{Z}_{3}=\{0,1,2\}$ and fix the ordered $5-$tuple $\mathbb{Z}^{2}_{3}=\{1,2\}.$
 Now the set of polynomials of degree at most one with constant term zero is $S^{3}=\{0,x,2x\}.$ Let us look at the arrangement of the polynomials in the construction of GES($3,2,2$).
\begin{displaymath}
\begin{bmatrix}
0 & x & 2x  \\
1 & x + 1 & 2x + 1 \\
2 & x + 2 & 2x + 2 
\end{bmatrix}.
\end{displaymath}
\noindent After evaluating the polynomials at $\mathbb{Z}^{2}_{3},$ we get GES($3,2,1$) as a $3\times 3$ square matrix with entries being $2-$tuples:
\begin{displaymath}
\begin{bmatrix}
(0,0) & (1,2) & (2,1)  \\
(1,1) & (2,0) & (0,2) \\
(2,2) & (0,1) & (1,0) 
\end{bmatrix}.
\end{displaymath}
\noindent One may observe that GES($3,2,1$) satisfies the properties of theorem~\ref{es_square_prime}.
\end{example}
\begin{remark}
Note that, the third condition of theorem~\ref{es_square_prime} is stronger than (ES 3) given in Section~\ref{ES_block_sparse}. Later, we use this property while calculating block coherence of binary matrices coming Euler Square.
\end{remark}

\subsection{GES($p,k,t$) as a rectangular  array}
\label{prime_rectangle_array}

%{\color{blue}{
%Even though Generalized Euler squares are introduced in this paper as higher dimensional analogues of Euler squares,we provide two dimensional configurations of the same objects for the purpose of applications. 
Consider a two dimensional matrix where the rows correspond to the $p$ different values of the dimension (corresponding to the degree $t$ in the polynomials) and the columns correspond to particular values for the other $t$ dimensions ($0,1,2,\cdots,t-1$). The entries of each column are the $k$-tuples obtained by evaluating the $p$ number of polynomials 
in the dimension $t$. The matrix obtained is of size $p\times p^{t}$, with each entry being a $k-$tuple, and satisfies (GES 1),(GES 2), (GES 3) and(GES 4).
%({\bf Can we illustrate one simple case diagrammatically or graphically or through example with t=2.. which will help convey message easily .. think of possibility}
%To bring little more clarity on the rectangular array structure of GES$(p,k,t)$,  we follow a special way of picking polynomials of degree at most $t.$ 

Consider the finite field
$\mathbb{F}_{p}=\{f_{0}=0,f_{1},\dots,f_{p-1}\}$
where $p$ is a prime or a prime power.
Let $S^{p}$ be the collection of polynomials of degree at most
$t$ (where $t < p-1$), with zero being the constant term. It is easy to
check that the cardinality
of $S^{p}$ is $|S^{p}|=p^{t}.$  For $P\in S^{p},$ define the set
$S^{p}_{P}=\{P_{j}=P+f_{j}:j=0, \dots, p-1\}$. Fix any ordered $k-$tuple
$z\in \mathbb{F}_p^k$ with $t<k\leq p-1$. For simplicity, we consider
$z=(f_{1}, \dots, f_{k})$.
An ordered $k-$tuple is formed
after evaluating $P_{j}$ at each of the points of $z$, that is, $d^{P}_{j} := \big(P_{j}(f_{1}), \cdots, P_{j}(f_{k})\big)$.
In order to make $d^{P}_{j}$ a $k-$tuple of numbers on 
$\{0,\cdots,p-1\},$ we replace $f_{i}$ with its index 
$i$.  Now $d^{P}_{j}$ for $j=0, \dots,p-1$ forms one column and as $|S^{p}|=p^{t}$ there are $p^{t}$ such columns. Therefore we get a matrix of  size $p\times p^{t}$ with each entry being a $k-$tuple.
%It may now be easily shown that rectangle obtained satisfies (GES 1),(GES 2), (GES 3) and(GES 4).
%Now from construction, it is easy to check that the matrix satisfies (GES 1), (GES 2), (GES 3) and (GES 4).

Let us take two $k-$tuples from same column. As they belong to same column, the corresponding polynomials are of the form $P+f_{i}$ and $P+f_{j}$ for $i\neq j$ and $P$ being a polynomial of degree at most $t$, with  zero being the constant term. Since  $P+f_{i}$ and $P+f_{j}$ does not share any common root, there is no intersection between the tuples coming from same column. 

Let us take two $k-$tuples from same row. As they belong to same row the corresponding polynomials are of the form $P_{1}+f_{i}$ and $P_{2}+f_{i}$ for $P_{1}$ and $P_{2}$ being a polynomial of degree at most $t$ with  zero being the constant term. Since  $P_{1}+f_{i}$ and $P_{2}+f_{i}$ share at most $t-1$ non zero common roots and $0$ as a common root, there is at most $t-1$ intersection between the tuples coming from same row as the polynomials are not evaluated at zero while forming the tuples.

Let us take two $k-$tuples from different row and column. As they belong to different row and column  the corresponding polynomials are of the form $P_{1}+f_{i}$ and $P_{2}+f_{j}$ for $i\neq j$ and $P_{1}$ and $P_{2}$ being a polynomial of degree at most
$r$ with  zero being the constant term. Since  $P_{1}+f_{i}$ and $P_{1}+f_{j}$ share at most $r$ non zero common root, there is at most $r$ intersection between the tuples coming from different row and column.
%}}

\noindent {\bf Example:} 
Let us first take $\mathbb{F}_{5}=\mathbb{Z}_{5}=\{0,1,2,3,4\}$ and fix the ordered $5-$tuple $\mathbb{Z}^{2}_{3}=\{1,2,3,4\}.$
Let us look at the arrangement of the polynomials in the construction of GES($5,4,2$): 
\begin{displaymath}
\begin{bmatrix}
0 & x & 2x & \dots & 4x^{2}+3x & 4x^{2}+4x \\
1 & x+1 & 2x+1 & \dots & 4x^{2}+3x+1 & 4x^{2}+4x+1\\
2 & x+2 & 2x+2 & \dots &  4x^{2}+3x +2& 4x^{2}+4x+2\\
3 & x+3 & 2x+3 & \dots & 4x^{2}+3x+3 & 4x^{2}+4x+3 \\
4 & x+4 & 2x+4 & \dots & 4x^{2}+3x+4 & 4x^{2}+4x+4 \\
\end{bmatrix}.
\end{displaymath}
Now, evaluating polynomials at $\mathbb{Z}^{2}_{3}=\{1,2,3,4\},$ we obtain GES($5,4,2$):  
\begin{displaymath} 
 \begin{bmatrix}
(0,0,0,0) & (1,2,3,4) & (1,4,2,3) & \dots &  (2,2,0,1) &  (3,4,3,0) \\
(1,1,1,1) & (2,3,4,0) & (2,0,3,4) & \dots & (3,3,1,2) & (4,0,4,1)\\
(2,2,2,2) & (3,4,0,1) & (3,1,4,0) & \dots &  (4,4,2,3)& (0,1,0,2)\\
(3,3,3,3) & (4,0,1,2) & (4,2,0,1) & \dots & (0,0,3,4) & (1,2,1,3) \\
(4,4,4,4) & (0,1,2,3) & (0,3,1,2) & \dots & (1,1,4,0) & (2,3,2,4) \\
\end{bmatrix}.~
\end{displaymath}
\subsection{GES($p'p'', k,t$) as rectangular array}
	%{\color{blue}{
	As in prime or prime power cases, one can arrange 
	$$\{(c_{m_{1},m_{2}, \ldots ,m_{t+1},1}, c_{m_{1},m_{2}, \ldots 
		,m_{t+1},2}\ldots,c_{m_{1},m_{2}, \ldots ,m_{t+1},k})\}^{p'p''}_{m_{s}=1}$$ as a rectangular matrix of size $p'p''\times 
	(p'p'')^{t}$ which satisfies the same properties given before. 
 We provide an equivalent form of GES($p'p'', k,t$) as a rectangular matrix of size $p'p''\times (p'p'')^{t}$ in the following way:  
 \par Following the construction in  subsection~\ref{prime_rectangle_array}, let $d^{P}_{i}$, for $i=1, \dots,(p')^{t}$,  form $i^{\text{th}}$ column of GES($p',k,t$) and $d^{Q}_{j}$, for $j=1, \dots,(p'')^{t}$, form the $j^{\text{th}}$ column of GES($p'',k,t$). Now a column $d^{P,Q}_{i,j}$ is formed by
	$$d^{P,Q}_{i,j}=d^{P}_{i}+(d^{Q}_{j}-\mathbf{1_{k}})p', $$ where $\mathbf{1_{k}}$ is a $k-$tuple of all ones. Since $d^{P}_{i}$ has $p'$ number of $k-$tuples and $d^{Q}_{j}$ has $p''$ number of $k-$tuples, one gets $p'p''$ such $k-$tuples in a column and each entry of $d^{P,Q}_{i,j}$ lies in $\{0,1, \dots, p'p''-1\}$ as each entry of $d^{P}_{i}$ lies in $\{0,1, \dots, p'-1\}$ and each entry of $d^{Q}_{j}$ lies in $\{0,1, \dots, p''-1\}.$  For every possible combination of $i$ and $j,$ we get $(p'p'')^{t}$ columns. Therefore, we can obtain a rectangle $D^{P,Q}$ of size $p'p'' \times (p'p'')^{t}$ with entry being a $k-$tuple and each entry of the $k-$tuples lies between $0$ and $p'p''-1.$ For $1\leq \ell \leq k,$ let us denote the $\ell^{\text{th}}$ entry of $s^{\text{th}}$ $k-$tuple of $d^{P}_{i},$ $\ell^{\text{th}}$ entry of $n^{\text{th}}$ $k-$tuple of $d^{Q}_{j}$ and $\ell^{\text{th}}$ entry of $m^{\text{th}}$ $k-$tuple of $d^{P,Q}_{i,j}$ as $d^{P}_{i}(s,\ell),$ $d^{Q}_{j}(n,\ell)$ and $d^{P,Q}_{i,j}(m,\ell)$, respectively, where $1\leq s \leq p',$ $1\leq n \leq p''$ and  $1\leq m \leq p'p''.$ Let $m_{1}\neq m_{2},$ then we get, $d^{P,Q}_{i,j}(m_{1},\ell)\neq d^{P,Q}_{i,j}(m_{2},\ell).$ It follows from the fact that $d^{P}_{i}(s_{1},\ell) \neq d^{P}_{i}(s_{2},\ell)$ and $d^{P}_{j}(n_{1},\ell) \neq d^{Q}_{j}(n_{2},\ell)$ for $s_{1}\neq s_{2}$ and $n_{1}\neq n_{2}.$ Hence $D^{P,Q}$ satisfies (GES 2). Using properties of GES$(p',k,t)$ and GES$(p'',k,t)$ and from construction procedure, it is easy to check that $D^{P,Q}$ satisfies (GES 3) and (GES 4) too. Hence, $D^{P,Q}$ forms a GES($p'p'',k,t$).
	%}}

\begin{theorem}
\label{property_ges}
	%\sout{Suppose $n = 2^{r}p^{r_{1}}_{1}p^{r_{2}}_{2}\ldots, p^{r_{l}}_{l}$ for distinct odd primes $p_{1},p_{2},\ldots, p_l$ 
%	and $t<k<\min\{2^{r},p^{r_{1}}_{1},p^{r_{2}}_{2},\ldots ,p^{r_{l}}_{l}\},$ then GES($n,k,t$) exists and it can be represented as an $n \times n^{t}$ matrix with each entry being a $k-$tuple of numbers taken from $\{1,2,\dots,n\}$ and has the following properties:} \\
	Suppose $n = 2^{r}p^{r_{1}}_{1}p^{r_{2}}_{2}\ldots, p^{r_{l}}_{l}$ for distinct odd primes $p_{1},p_{2},\ldots, p_l$ 
	and $t<k<\min\{2^{r},p^{r_{1}}_{1},p^{r_{2}}_{2},\ldots ,p^{r_{l}}_{l}\}$. Then GES($n,k,t$) exists, which can be represented as an $n \times n^{t}$ matrix with each entry being a $k-$tuple of numbers taken from $\{1,2,\dots,n\}$ and has the following properties:
	\begin{itemize}
	\item Two distinct $k-$tuples from the same column do not intersect.
	\item Two distinct $k-$tuples from same row have at most $t-1$ intersections.
	\item Any two distinct $k-$tuples in the array can have at most $t$ intersections.
\end{itemize}

\end{theorem}

%%%%%%%%%%%%%%%%%%%%%%%%%%%%%%%%%%%%%%%%%%%%%%%%%%%%%%%%%%%%%%%%%%%%%%%%%%%%%%%%%%%%%%%%%%%%%%%%%%%%%%%%%%%%%%%%%%%%%%%%%%%%%%%%%%%
%\begin{remark}
%For $p$ being a prime or prime power, in \cite{Ronald_2007}, the author has constructed matrices of size $p^{2}\times p^{r+1}$ with coherence $\frac{r}{k}.$ This class of matrices is GES matrix $\Phi^{p}_{p_{r}}$ of index $p,p,r.$  
%\end{remark}
%\begin{remark}
%IThe following proposition provides a result on improvement of column size of the GES based CS matrices further. 
%\begin{proposition}{Column extension of GES matrix}\label{subsec:col}
%For GES based CS matrix, 
%By applying the methodology proposed in \cite{ram_2016}, one may extend the column size of GES based CS matrices and obtain a class of matrices of size  
%for extending the columns of matrices generated from Euler Square, we obatin a CS matrix of size
%$m\times C_{m}(\frac{m\mu}{r})^{r+1}$, where $\mu$ is the coherence of the matrix and $r$ is smaller than the smallest prime power factor of $m$ and $C_{m}\in [1,2).$
%\end{proposition}
%\end{remark}
%%%%%%%%%%%%%%%%%%%%%%%%%%%%%%%%%%%%%%%%%%%%%%%%%%%%%%%%%%%%%%%%%%%%%%%%%%%%%%%%%%%%%%%%%%%%%%%%%%%%%%%%%%%%%%%%%%%%%%%%%%%%%%%%%%%%%%%%%%%%%%%%%%%%%%%%%%%%%%%%%%%%%%%%%%%%%%%%%%%%%%%%%%%%%%%%%%%%%%%%%%%%%%%%%%%%%%%%%%%%%%%%%%%%%%%%%%%%%%%%%%%%%%%%%%%%%%%%%%%%%%%%%%%%%%%%%%%%%%%%%%%

\section{Recovery guarantees for block sparse signals via GES}
%In previous section, we have not used the properties of generalized Euler square given in Theorem~\ref{property_ges}. 
In this section, making use of the properties of GES, we show that the binary matrices constructed from GES are capable of recovering block sparse signals.

\subsection{Block orthogonality}
%We arrange the columns of the binary matrix obtained from GES and obtain a block orthogonal structure in it. 

\begin{theorem}
\label{ges_block_sparse}
If GES$(n,k,t)$ exists, then a sparse matrix of size $nk \times n^{t+1}$ exists, which consists of $n^{t}$ orthonormal blocks, each of size $nk\times n$.
\end{theorem}
\begin{proof}
The binary matrix $\Phi(n,k,t)$ of size $nk \times n^{t+1}$ is obtained from GES($n,k,t$). Every column of $\Phi(n,k,t)$ corresponds to a unique $k-$tuple of GES($n,k,t$). We arrange the columns of $\Phi(n,k,t)$ to form a block orthogonal matrix. We form the $\ell-$th block (of block size $nk \times n$) by taking $n$ columns of $\Phi(n,k,t)$ corresponding to $n$ $k-$tuples coming from $\ell-$th column of GES($n,k,t$).  From Theorem~\ref{property_ges}, we know that two $k-$tuples belonging to same column of a GES do not have any intersection. As a result, the inner product between any two different columns within a block is zero, implying thereby that each block is orthogonal. Consequently, the block matrix  $\frac{1}{\sqrt{k}}\Phi(n,k,t)$ has $n^{t}$ orthonormal blocks, where each block is of size $nk\times n.$ 
\end{proof}

%\begin{theorem}
%The block coherence of %$\frac{1}{\sqrt{k}}\Phi^{n}_{k_{r}}$ is at most $\frac{1}{n}.$
%\end{theorem}
%\begin{proof}
%The proof follows from Theorem~\ref{ges_block_sparse} and Proposition~\ref{block_coherence}.
%\end{proof}

In the case of generalized Euler square, the conditions given in Theorem~\ref{thm:bomp} for the successful recovery of block sparse signal of block size $d$ via BOMP makes sense provided the block coherence of $\frac{1}{\sqrt{k}}\Phi(n,k,t)$ is strictly less than $\frac{1}{d}.$ In view of this, our next objective is to choose  $n,k$ and $d$ such that the block coherence of $\frac{1}{\sqrt{k}}\Phi(n,k,t)$ becomes strictly less than $\frac{1}{d}.$ For simplicity,  we first establish the block coherence of Euler square matrices.

\subsection{Block coherence of Euler Square}
%In the case of Euler square, the conditions given in Theorem~\ref{thm:bomp} for the successful recovery of block sparse signal of block size $nk \times d$ via BOMP make sense provided the block coherence of $\frac{1}{\sqrt{k}}\Phi^{n}_{k}$ is strictly less than $\frac{1}{d}.$  In view of this, our next objective is to choose the block size  $nk \times d$ such that the block coherence of $\frac{1}{\sqrt{k}}\Phi^{n}_{k}$ becomes strictly less than $\frac{1}{d}.$ 
%\subsection{Block sparse signal recovery for GES matrices}

%%%%%%%%%%%%%%%%%%%%%%%%%%%%%%%%%%%%%%%%%%%%%%%%%%%%%%%%%%%%%%%%%%%%%%%%%%%%%%%%%%%%%%%%%%%
\begin{theorem}
\label{es_block_coherence}
Suppose $p$ is a prime or a power of prime, $d\leq p-1$ and $d$ divides $p.$ Then a binary matrix of size $p(p-1) \times p^{2}$ with block coherence $\frac{1}{p-1}$ exists, which consists of $\frac{p^{2}}{d}$ number of orthonormal blocks, each of size $p(p-1) \times d$. %\sout{The block coherence of the matrix is $\frac{d-1}{dk}.$} 
\end{theorem}
\begin{proof}
Recall that ES($n,k$) is same as GES($n,k,1$). Hence, from theorem~\ref{ges_block_sparse}, we get the binary matrix $\frac{1}{\sqrt{k}}\Phi(n,k,1)$ consisting of $\frac{n^{2}}{d}$ number of orthonormal blocks, where each block is of size $nk \times d.$  Now take, $k=p-1.$
Let $\frac{1}{\sqrt{k}}\Phi(n,k,1)[\ell]$ denote the $\ell^{th}$ block of $\frac{1}{\sqrt{k}}\Phi(n,k,1).$
 From the construction of Euler square described in Theorem~\ref{es_square_prime} %section~\ref{sec:new_construction} 
 and the properties of Euler square, it follows, for $\ell \neq q,$ that

(i)% \textcolor{red}{
When $\Phi(n,k,1)[\ell]$ and $\Phi(n,k,1)[q]$ correspond to $k-$tuples coming from different columns but from the same rows of ES($n,k$), we have
%}

\begin{equation*}
    (\frac{1}{\sqrt{k}}\Phi(n,k,1)[\ell])^{T}(\frac{1}{\sqrt{k}}\Phi(n,k,1)[q])= \frac{1}{k}\begin{pmatrix}
        0 & 1 & 1 & 1 &  \cdots & 1  \\
        1 & 0 & 1 & 1 & \cdots & 1 \\
        1 & 1 & 0 & 1 & \cdots & 1  \\
        \vdots & \vdots & \vdots & \ddots & \vdots & \vdots \\
        1 & 1 & 1 & \cdots & 1 & 0  
     \end{pmatrix}_{d\times d},
\end{equation*}
that is, all diagonal entries of $(\Phi(n,k,1)[\ell])^{T}\Phi(n,k,1)[q]$ are zero and all off-diagonal entries are one. 
%\sout{This happens when $\Phi(n,k,1)[\ell]$ and $\Phi(n,k,1)[q]$ correspond to $k-$tuples coming from different columns but from same rows of ES($n,k$)}. 
The maximum eigen value of $(\frac{1}{\sqrt{k}}\Phi(n,k,1)[\ell])^{T}(\frac{1}{\sqrt{k}}\Phi(n,k,1)[q])$ is $\frac{d-1}{k}.$

ii) %\textcolor{red}{
When $\Phi(n,k,1)[\ell]$ and $\Phi(n,k,1)[q]$ correspond to $k-$tuples coming from different columns and rows of ES($n,k$), we have
%}

\begin{equation*}
    (\frac{1}{\sqrt{k}}\Phi(n,k,1)[\ell])^{T}(\frac{1}{\sqrt{k}}\Phi(n,k,1)[q])= \frac{1}{k}\begin{pmatrix}
        1 & 1 & 1 & 1 &  \cdots & 1  \\
        1 & 1 & 1 & 1 & \cdots & 1 \\
        1 & 1 & 1 & 1 & \cdots & 1  \\
        \vdots & \vdots & \vdots & \ddots & \vdots & \vdots \\
        1 & 1 & 1 & \cdots & 1 & 1 
     \end{pmatrix}_{d\times d},
\end{equation*}
that is,  $(\Phi(n,k,1)[\ell])^{T}\Phi(n,k,1)[q]$ is an all one matrix. %\sout{This happens when $\Phi(n,k,1)[\ell]$ and $\Phi(n,k,1)[q]$ correspond to $k-$tuples coming from different columns and rows of ES($n,k$)}.
The maximum eigen value of $(\frac{1}{\sqrt{k}}\Phi(n,k,1)[\ell])^{T}(\frac{1}{\sqrt{k}}\Phi(n,k,1)[q])$ is $\frac{d}{k}.$

(iii) %\textcolor{red}{
When $\Phi(n,k,1)[\ell]$ and $\Phi(n,k,1)[q]$ correspond to $k-$tuples coming from the same column of ES($n,k$), we have
%}
\begin{displaymath}
(\frac{1}{\sqrt{k}}\Phi(n,k,1)[\ell])^{T}(\frac{1}{\sqrt{k}}\Phi(n,k,1)[q])=\mathbf{0}.
\end{displaymath}
%\sout{when $\Phi(n,k,1)[\ell]$ and $\Phi(n,k,1)[q]$ correspond to $k-$tuples coming from same column of ES($n,k$).} 
Here, $\mathbf{0}$ denotes the zero matrix.

\noindent Consequently, the block coherence $\mu_{B_{\frac{1}{\sqrt{k}}\Phi(n,k,1)}}$ is $\frac{d}{dk}=\frac{1}{k}.$  %\textcolor{red}{(coherence is equal to OR at most  $\frac{d-1}{dk}.$ .. I think equality due to Prop. 7.4. )}

\end{proof}

\begin{remark}
The block coherence of $\frac{1}{\sqrt{p-1}}\Phi(p,p-1,1)$ is at most $\frac{1}{p-1}$, which can also be obtained from the fact that the coherence of $\Phi(p,p-1,1)$ is at most $\frac{1}{p-1}.$ The significance of the Theorem~\ref{es_block_coherence} is that it uses the fact that there is exactly one intersection between any two distinct $(p-1)-$tuples which are not from the same row or column and proves that the block coherence of $\frac{1}{\sqrt{p-1}}\Phi(p,p-1,1)$ is exactly $\frac{1}{p-1}.$
\end{remark}

\begin{theorem}
\label{ES_Block_coherence_Gen}
Suppose, GES($n,k,1$) exists, then for $d< k$ and $n$ being a multiple of $d$, 
%$d\leq \lfloor \frac{k}{r} \rfloor$ 
 a sparse matrix of size $nk \times n^{2}$ exists which consists of $\frac{n^{2}}{d}$ orthonormal blocks, each of size $nk\times d.$  Then the block coherence of $\frac{1}{\sqrt{k}}\Phi(n,k,1)$ is at most $\frac{1}{k}.$
\end{theorem}
\begin{proof}
The proof follows from the Theorem~\ref{ges_block_sparse} and the fact that the coherence of  $\frac{1}{\sqrt{k}}\Phi(n,k,1)$ is at most $\frac{1}{k}$.
\end{proof}

\noindent Now from Theorem~\ref{thm:bomp}, the following theorem follows immediately.
\begin{theorem}
\label{es_block_sparse_recovery}
A sufficient condition for the BOMP to recover a block $s-$sparse signal $x_{0}$, with each block length $ d$, from $y=\frac{1}{\sqrt{k}}\Phi(n,k,1) x_{0}$ is $$s<\frac{1}{2}\bigg(1+\frac{k}{d}\bigg),$$ provided $d\leq k.$ 
\end{theorem}

\subsection{Block coherence of GES matrices}
Now we derive the expression for block coherence of binary matrices constructed from GES.
\begin{theorem}
Suppose, GES($n,k,t$) exists, then for 
$d\leq \lfloor \frac{k}{t} \rfloor$  
 and $n$ being a multiple of $d$, a sparse matrix of size $nk \times n^{t+1}$ exists which consists of $\frac{n^{t+1}}{d}$ orthonormal blocks, each of size $nk\times d.$ Then
the block coherence of $\frac{1}{\sqrt{k}}\Phi(p,p-1,t)$ is at most $\frac{t}{k}.$ 
\end{theorem}
\begin{proof}
The proof follows from the Theorem~\ref{ges_block_sparse} and the fact that the coherence of  $\frac{1}{\sqrt{k}}\Phi(n,k,r)$ is at most $\frac{r}{k}$.
\end{proof}

\noindent Now from Theorem~\ref{thm:bomp}, the following theorem follows immediately.

\begin{theorem}
\label{Ges_block_sparse_recovery}
A sufficient condition for the BOMP to recover a block $s-$sparse signal $x_{0}$, with $nk \times d$ as size of each block, from $y=\frac{1}{\sqrt{k}}\Phi(n,k,r) x_{0}$ is $$s<\frac{1}{2}\bigg(1+\frac{k}{dr}\bigg),$$ provided $d\leq \lfloor \frac{k}{r} \rfloor.$
\end{theorem}

\section{Concluding Remarks : }
In our present work, we have constructed block orthogonal binary matrices via Euler Squares with low block coherence which supports recovery of block sparse signals. We have also introduced 
%a generalization of Euler square, namely 
%Generalized Euler Squares (GES).  We have presented Euler Squares with a different perspective and associated them  with polynomials of degree at most one over a finite field. 
and constructed Generalized Euler Squares (GES) by evaluating higher degree polynomials over a finite field. The binary matrices constructed from GES have been shown to possess better aspect ratio compared to their counterparts generated using Euler Squares.
%\item bringing significant increment in \sout{terms of} column size compared to the ones obtained via Euler Squares
Finally, block orthogonal structure of the GES based binary matrices has been established.

%\section{An simple approach to construct ternary matrix:}
%In this section, we present another deterministic construction procedure of ternary matrix by combining binary and hadamard matrices . We show that the resulting matrix has same properties as binary but with better aspect ratio. Let $\Psi_{m \times M}$ be a binary sensing  matrix having $k$ number of one's in each column and the overlap between any two columns is $r$. Suppose there exist a hadamard matrix $H$ of size $(k+r')$ for some $r' \in { \{0,\dots, r\}}$. Now to get a ternary CS matrix $\Phi$,  for each column of $\Psi$ we replace each of
%its $1$-valued entries with a distinct row of $H$. So the size of the matrix $\Phi$ becomes $m \times M(k+r')$. As the rows of hadamard matrix $H$ are orthogonal, the rows of $\Phi$ are orthogonal. By the construction it is very easy to check that the coherence of the matrix $\Phi$ is $\frac{r}{k}$. The density of the matrix $\Phi$ is $\frac{k}{m}$.

%%%%%%%%%%%%%%%%%%%%%%%%%%%%%%%%%%%%%%%%%%%%%%%%%%%%%%%%%%%%%%%%%%%%%%%%%%%%%%%%%%%%%%%%%%%%%%%%%%%%%%%%%%%%%%%%%%%%%%%%%%%%%%%%%%%%%%%%%%%%%%%%%%%%%%%%%%%%%%%%%%%%%%%%%%%%%%%%%%%%%%%%%%%%%%%%%%%%%%%%%%%%%%%%%%%%%%%%%%%%%%%%%%%%%%%%%%%%%%%%%%%%%%%%%%%%%%%%%%%%%%%%%%%%%%%%%%%%%%%%%%%

\section{\bf Acknowledgments}
The first author is thankful for the support that he receives from Science and Engineering Research Board (SERB), Government of India (PDF/2017/002966).
 %\section{References} 
 
\end{document}